\documentclass[twoside,11pt,reqno]{amsart}
\usepackage{amsmath,amssymb,amscd,mathrsfs, todonotes,appendix}
\usepackage{graphics,verbatim}
\usepackage{todonotes}
\usepackage{enumitem}
\usepackage{hyperref}
\usepackage{setspace} 

\usepackage{xcolor}

\usepackage{float} 
\usepackage[labelformat=empty]{caption}
\usepackage{fancybox} 

\newcommand{\losemi}{{\otimes \kern -.78em \ltimes}}
\newcommand{\rosemi}{{\otimes \kern -.78em \rtimes}}

\newcommand{\Hom}{\ensuremath{\operatorname{Hom}}}

\newcommand{\Ext}{\operatorname{Ext}}

\newcommand{\ind}{\operatorname{ind}}

\renewcommand{\a}{\alpha}

\newcommand{\ga}{\gamma}
\newcommand{\s}{\sigma}

\newcommand{\la}{\lambda}

\newcommand{\Mod}{\operatorname{Mod}}

\newcommand{\St}{\operatorname{St}}

\newcommand{\opH}{\operatorname{H}}

\newcommand{\hQ}{\widehat{Q}}
\newcommand{\hZ}{\widehat{Z}}

\makeatletter
\newcommand{\leqnomode}{\tagsleft@true}
\newcommand{\reqnomode}{\tagsleft@false}
\makeatother

\newtheorem{theorem}{Theorem}[subsection]

\makeatletter\let\c@fact\c@theorem\makeatother

\makeatletter\let\c@note\c@theorem\makeatother

\makeatletter\let\c@lemma\c@theorem\makeatother

\makeatletter\let\c@lemma\c@theorem\makeatother

\newtheorem{quest}{Question}[subsection]
\makeatletter\let\c@quest\c@theorem\makeatother

\newtheorem{prop}{Proposition}[subsection]
\makeatletter\let\c@prop\c@theorem\makeatother

\newtheorem{conj}{Conjecture}[subsection]
\makeatletter\let\c@conj\c@theorem\makeatother

\makeatletter\let\c@cor\c@theorem\makeatother

\makeatletter\let\c@defn\c@theorem\makeatother

\theoremstyle{definition}

\makeatletter\let\c@remark\c@theorem\makeatother

\makeatletter\let\c@example\c@theorem\makeatother
\numberwithin{equation}{subsection}

\usepackage[capitalise]{cleveref}
\crefname{theorem}{Theorem}{Theorems}
\crefname{fact}{Fact}{Facts}
\crefname{note}{Note}{Notes}
\crefname{lemma}{Lemma}{Lemmas}
\crefname{alg}{Algorithm}{Algorithms}
\crefname{remark}{Remark}{Remarks}
\crefname{example}{Example}{Examples}
\crefname{prop}{Proposition}{Propositions}
\crefname{conj}{Conjecture}{Conjectures}
\crefname{cor}{Corollary}{Corollaries}
\crefname{defn}{Definition}{Definitions}
\crefname{equation}{\!\!}{\!\!} 

\newcounter{listequation}

\begin{document}

\title{On Donkin's Tilting Module Conjecture III: New Generic Lower Bounds}

\dedicatory{In memory of Georgia M. Benkart.}

\begin{abstract}  In this paper the authors consider four questions of primary interest for the representation theory of reductive algebraic groups: 
(i) Donkin’s Tilting Module Conjecture, (ii) the Humphreys-Verma Question, (iii) whether $\text{St}_r \otimes L(\lambda)$ is a tilting module for
$L(\lambda)$ an irrreducible representation of $p^{r}$-restricted highest weight, and (iv) whether $\text{Ext}^{1}_{G_{1}}(L(\lambda),L(\mu))^{(-1)}$ is a tilting module where $L(\lambda)$ and $L(\mu)$ have $p$-restricted highest weight.   

The authors establish affirmative answers to each of these questions with a new uniform bound, namely $p\geq 2h-4$ where $h$ is the Coxeter number. 
Notably, this verifies these statements for infinitely many more cases. Later in the paper, questions (i)-(iv) are considered for rank two groups where there are counterexamples (for small primes) to these questions. 
\end{abstract}

\author{\sc Christopher P. Bendel}
\address
{Department of Mathematics, Statistics and Computer Science\\
University of
Wisconsin-Stout \\
Menomonie\\ WI~54751, USA}
\thanks{Research of the first author was supported in part by Simons Foundation Collaboration Grant 317062}
\email{bendelc@uwstout.edu}

\author{\sc Daniel K. Nakano}
\address
{Department of Mathematics\\ University of Georgia \\
Athens\\ GA~30602, USA}
\thanks{Research of the second author was supported in part by
NSF grants DMS-1701768 and DMS-2101941}
\email{nakano@math.uga.edu}

\author{\sc Cornelius Pillen}
\address{Department of Mathematics and Statistics \\ University
of South
Alabama\\
Mobile\\ AL~36688, USA}
\email{pillen@southalabama.edu}

\author{Paul Sobaje}
\address{Department of Mathematical Sciences \\
          Georgia Southern University\\
          Statesboro, GA~30458, USA}
\email{psobaje@georgiasouthern.edu}

\maketitle
\section{Introduction}

\subsection{} For modular representations of Lie algebras that arise from an algebraic group, an important problem has been to determine when representations lift to the ambient algebraic group. 
In 1960, Curtis \cite{C60} showed that when $G$ is a simple simply connected algebraic group over an algebraically closed field $k$ of positive characteristic $p$, the simple restricted representations for 
${\mathfrak g}=\text{Lie }G$ lift uniquely to $G$. This result was central to the development of the theory of reductive algebraic groups, because Steinberg \cite{St63} was able to prove that 
irreducible rational $G$-modules can be constructed via twisted tensor products of simple ${\mathfrak g}$-modules (lifted to $G$). 

Restricted representations for the Lie algebra ${\mathfrak g}$ are equivalent to representations for the first Frobenius kernel $G_{1}$, and one can employ higher Frobenius kernels $G_{r}$ in the 
study of representations for $G$. The modules for $G_{r}$ coincide with modules for a finite-dimensional cocommutative Hopf algebra $\text{Dist}(G_{r})$, and one can consider the projective covers 
(equivalently, injective hulls) of simple $G_{r}$-modules. The Humphreys-Verma Question ([HV-Quest]) (see \cite[10.4 Question]{Hum06}), posed in 1973, asked whether the $G_r$-structure on such a projective module lifts to a $G$-structure. 
Ballard \cite{B78} provided an affirmative answer to this question for $p\geq 3h - 3$, and shortly thereafter Jantzen \cite{Jan80} lowered the bound to $p \geq 2h-2$. 

In 1990, Donkin \cite{Don93} proposed his famous Tilting Module Conjecture ([DTilt]) which states 
that such structures should arise from tilting modules for $G$.  This was similarly shown to hold for $p \geq 2h-2$. For many years, people believed that [DTilt] should hold for all $p$. 
The authors \cite{BNPS20} discovered the first counterexample to Donkin's Tilting Module Conjecture in 2019 and have subsequently produced counterexamples in the root systems $\rm{B}_{n}$ ($n\geq 3$), $\rm{C}_{n}$ ($n\geq 3$), $\rm{D}_{n}$ ($n\geq 4$), $\rm{E}_{6}$, $\rm{E}_{7}$, $\rm{E}_{8}$, $\rm{F}_{4}$, and $\rm{G}_{2}$ (see \cite{BNPS22b}). For types $\rm{A}_{n}$ ($n=1,2,3$) and $\rm{B}_{2}$, [DTilt] holds for all primes. No counterexamples to [DTilt] have been found in type $\rm{A}_n$. The methods employed for constructing counterexamples involved using the structure of $\text{Ext}^{1}_{G_{1}}(L(\lambda),L(\mu))^{(-1)}$ for $p$-restricted weights $\lambda$ and $\mu$ and using the fact that [DTilt] implies that these cohomology groups must embed in a tilting module. The fact that $\text{Ext}^{1}_{G_{1}}(L(\lambda),L(\mu))^{(-1)}$ is a tilting module was shown by Andersen \cite{And84} for $p\geq 3h - 3$. This bound was subsequently lowered to $2h-2$ by Bendel, Nakano and Pillen \cite{BNP04}. 

Donkin also in 1990 proposed another conjecture that was based on a question posed by Jantzen in 1980 \cite{Jan80}. Donkin's $(p,r)$-Filtration Conjecture ([DFilt $\Leftrightarrow$]) provides a necessary and sufficient condition for a 
rational $G$-module to admit a good $(p,r)$-filtration. Much of the recent progress made by the authors in verifying and disproving the Tilting Module Conjecture in specific cases relied heavily on results by 
Kildetoft and Nakano, and Sobaje who made important connections between [DTilt] and good $(p,r)$-filtrations on $G$-modules (cf. \cite{KN15}, \cite{So18}). For a detailed description of these interrelationships, the reader is referred to Section~\ref{S:conjquest}.

\subsection{}\label{S:main} The main goals of the paper are (i) to make comparisons between the aforementioned questions and (ii) to establish a new uniform bound, namely $p\geq 2h-4$ for the 
validity of the problems (a)-(d) in Theorem \ref{T:main} below.  For parts (a) and (b), the bound $p\geq 2h-2$ was the best bound available for over 30 years, and for part (d) this was the best known bound for the past 15 years. 
The bound for part (c) was lowered to $2h-4$ recently in \cite{BNPS19}. In this paper, we note that our results allows us to verify (a), (b), and (d) for infinitely many more cases. Our results also provide some hope that the bounds can be further improved to perhaps $p>h$. 

\begin{theorem}\label{T:main} Let $G$ be a simple algebraic group over an algebraically closed field of characteristic $p>0$ and $h$ be the Coxeter number associated to the root system for $G$. Then the following hold for $p\geq 2h-4$: 
\begin{itemize} 
\item[(a)]  Donkin's Tilting Module Conjecture [DTilt],
\item[(b)]  An affirmative answer to the Humpheys-Verma Question [HV-Quest],
\item[(c)] $\operatorname{St}_{r}\otimes L(\lambda)$ is a tilting module for all $\lambda\in X_{r}$ (i.e., [DFilt $\Rightarrow$]),
\item[(d)] $\operatorname{Ext}^{1}_{G_{1}}(L(\lambda),L(\mu))^{(-1)}$ is a tilting $G$-module for all $\lambda,\mu\in X_{1}$. 
\end{itemize} 
\end{theorem} 

For comparison, we also analyze the rank 2 cases in finer detail in an effort to gain further insight into these conjectures and their connections. 

\begin{theorem}\label{T:main2} Let $G$ be a simple algebraic group over an algebraically closed field of characteristic $p>0$ with underlying root system of rank $2$. Then 
\begin{itemize} 
\item[(a)]  Donkin's Tilting Module Conjecture [DTilt] holds for all $p$ except when $\Phi=G_{2}$ and $p=2$. 
\item[(b)]  The Humphreys-Verma Question [HW-Quest] has a positive answer for all $p$ except possibly when $\Phi=G_{2}$ and $p=2$. 
\item[(c)] $\operatorname{St}_{r}\otimes L(\lambda)$ is a tilting module for all $\lambda\in X_{r}$ (i.e., [DFilt $\Rightarrow$] holds) for all $p$.
\item[(d)] $\operatorname{Ext}^{1}_{G_{1}}(L(\lambda),L(\mu))^{(-1)}$ is a tilting $G$-module for all $\lambda,\mu\in X_{1}$ for all $p$ except when $\Phi=G_{2}$ and $p=2$. 
\end{itemize} 
\end{theorem} 

Even though there is a counterexample to Donkin's TMC for $\Phi=G_{2}$ $p=2$, the Humphreys-Verma Question still remains open in this case. 

\subsection{} The paper is organized as follows. In Section~\ref{S:prelim}, the conventions for the paper are established. Furthermore, a complete description of the problems stated in parts (a)-(d) of 
Theorem~\ref{T:main} is presented with an up-to-date summary of the known connections. The following section (Section~\ref{S:bound}) focuses on lowering the bound for the Tilting Module Conjecture. 
Section~\ref{S:Ext} establishes the lower bound for part (d) of Theorem~\ref{T:main}. Our first main theorem (Theorem~\ref{T:main}) can then be deduced from the results in \cite{BNPS20}, along with work in 
Sections \ref{S:prelim} and \ref{S:Ext}. Finally, in Section~\ref{S:Rank2}, we provide an analysis of the rank two cases that entails gathering known results from the literature and using prior work in 
\cite{BNPS22a}. 


\section{Preliminaries}\label{S:prelim}

\subsection{Notation}\label{S:notation} In this paper we will generally follow the standard conventions in \cite{rags}. Throughout this paper $k$ is an algebraically closed field of characteristic $p>0$. Let
\vskip .25cm 
\begin{itemize}
\item[(1)] $G$ be a connected semisimple algebraic group scheme defined over ${\mathbb F}_{p}$.
\item[(2)] $T$ be a fixed split maximal torus in $G$. 
\item[(3)] $\Phi$ be the root system associated to $(G,T)$. 
\item[(4)] $\Phi^{\pm}$ be the set of positive (resp. negative) roots. 
\item[(5)] $\Delta=\{\alpha_1,\dots,\alpha_{l}\}$ be the set of simple roots determined by $\Phi^+$. 
\item[(6)] $B$ be the Borel subgroup given by the set of negative roots, $U$ be the unipotent radical of $B$. 
\item[(7)] More generally, if $J\subseteq \Delta$, let $P_{J}$ be the
parabolic subgroup relative to $-J$, $L_{J}$ be the Levi factor of $P_{J}$ and $U_{J}$ be the
unipotent radical.
\item[(8)]  $\Phi_{J}$ be the root subsystem in $\Phi$ generated by the simple roots in $J$, 
with positive subset $\Phi_{J}^{+} = \Phi_{J}\cap\Phi^{+}$. 
\item[(9)] $W$ be the Weyl group associated with $\Phi$. For any $J\subseteq \Delta$, let $W_{J}$ be the subgroup of $W$ generated by reflections corresponding to simple 
roots in $J$.
\item[(10)]  $w_0$ (resp. $w_{J,0}$) denote the longest word of $W$ (resp. $W_{J}$, for $J \subseteq \Delta$). 
\item[(11)] $\rho$ be the half-sum of positive roots (which is also the sum of the fundamental weights). More generally, let $\rho_J$ be the half-sum of all the roots spanned by $J$.
\item[(12)] $\alpha_0$ be the highest short root, with associated coroot $\alpha_0^{\vee}$.
\item[(13)] $h$ denote the Coxeter number for the root system associated to $G$, i.e., $h = \langle\rho,\alpha_0^{\vee}\rangle + 1$.
\item[(14)] $X:=X(T)$ be the integral weight lattice spanned by the fundamental weights $\{\omega_1,\dots,\omega_l\}$. 
Moreover, for $J\subseteq \Delta$, let $X_{J}^{+}$ be the weights that are dominant on $J$ and $(X_J)_r$ denote the weights that are $p^r$-restricted on $J$. 
\item[(15)] $X^{+}$ denote the dominant weights for $G$. 
\item[(16)]  $X_{r}$ be the $p^{r}$-restricted weights. 
\item[(17)] $\leq$ be the order relation defined on $X$ via $\mu \leq \la$ iff $\la - \mu = \sum_{\a \in \Delta}n_{\a}\a$ for $n_{\a} \in {\mathbb Z}_{\geq 0}$.
\item[(18)] $\uparrow$ be the refined order relation on $X$ (associated to the Strong Linkage Principle) as defined in \cite[\S II.6.4]{rags}.
\end{itemize} 
\noindent 
For $\lambda\in X^{+}$, there are four fundamental families of finite-dimensional rational $G$-modules: 
\vskip .15cm 
$\bullet$ $L(\lambda)$ (simple), 
\vskip .15cm
$\bullet$ $\nabla(\lambda)$ (costandard/induced), 
\vskip .15cm 
$\bullet$
$\Delta(\lambda)$ (standard/Weyl), 
\vskip .15cm 
$\bullet$ 
$T(\lambda)$ (indecomposable tilting).
\vskip .25cm 
Let $F^{r}:G\rightarrow G$ be the $r$th iteration of the  Frobenius morphism, and $G_{r}$ be the scheme theoretic kernel of this map which is often called the 
$r$th Frobenius kernel. Set $G_{r}T=(F^{r})^{-1}(T)$. The category of $G_{r}$-modules is equivalent to the category of modules for a finite-dimensional cocommutative Hopf algebra. 
Let $Q_{r}(\lambda)$ denote the projective cover (equivalently, injective hull) of $L(\lambda)$ as a $G_{r}$-module, $\lambda\in X_{r}$.  For $\la \in X_r$, we will use the 
notation $\widehat{Q}_{r}(\lambda)$ to denote the lift of $Q_{r}(\lambda)$ to $G_{r}T$. Let  $\text{St}_r = L((p^r-1)\rho)$ be the $r$th Steinberg module which is a simple $G$-module and 
also projective/injective when restricted to $G_{r}T$ or $G_{r}$.  One further set of modules will be needed: baby Verma modules.  For $\la \in X$, set $\hZ'(\la) := \ind_{B}^{G_rB}\la$.

Given a Levi subgroup $L_J$ of $G$, each of the fundamental modules discussed above have a version defined over $L_J$ or $(L_J)_r$ as appropriate. Such modules will be denoted with a $J$-subscript. For example, $T_J(\sigma)$ will denote the indecomposable tilting $L_J$-module with highest weight $\sigma$.   

\subsection{Summary of the Questions-Conjectures}  \label{S:conjquest} We provide now a brief description of the conjectures and questions raised in the introduction in chronological order. 
For the reader's convenience, a chart describing the various connections and implications is presented at the end of the section. 
We start with a statement of the Humphreys-Verma Question which is still an open problem for all primes $p>0$. 

\vskip .5cm 
\begin{quest}\label{Q:HV} $\operatorname{\sf [HV-Quest]}$ For $\lambda\in X_{r}$, does $Q_{r}(\lambda)$ admit a structure as a $G$-module? 
\end{quest} 

Jantzen posed the following question in the early 1980s that identifies fine structures within standard/costandard modules. The general question is stated as part (a). 
A weaker version involving specific weights above the restricted region is stated as part (b).  A finite dimensional rational $G$-module $M$ is said to have a {\em good $(p,r)$-filtration} if there exists a filtration
$$
0 = F_0 \subseteq F_1 \subseteq F_2 \subseteq \cdots \subseteq F_n = M,
$$
such that, for $1 \leq i \leq n$, $F_{i}/F_{i-1} \cong L(\la_i)\otimes\nabla(\s_i)^{(r)}$ for weights $\la_i \in X_r$ and $\s_i \in X^+$.
\vskip .5cm 
\begin{quest}\label{Q:JQ} Let $G$ be a semisimple algebraic group over $k$. 
\begin{itemize} 
\item[(a)] $\operatorname{\sf [J-Quest]}$ For $\lambda\in X^{+}$, does $\nabla(\lambda)$ admit a good $(p,r)$-filtration? 
\item[(b)] $\operatorname{\sf [J-Quest (\dagger)]}$ For $\lambda\in X_{r}$, does $\nabla(2(p^r - 1)\rho + w_0\la)$ admit a good $(p,r)$-filtration? 
\end{itemize}
\end{quest} 

The famous Tilting Module Conjecture introduced by Donkin at MSRI  in 1990 states that the Humphreys-Verma Question has a positive 
answer via indecomposable tilting $G$-modules. 
 
\begin{conj}\label{C:TMC} $\operatorname{\sf [DTilt]}$ For all $\lambda\in X_{r}$, 
$$T((p^{r}-1)\rho+\lambda)|_{G_{r}T}\cong\widehat{Q}_{r}((p^{r}-1)\rho+w_{0}\lambda).$$
Equivalently, $T(2(p^{r}-1)\rho+w_{0}\lambda)|_{G_{r}T}\cong \widehat{Q}_{r}(\lambda)$ for all $\lambda\in X_{r}$. 
\end{conj}

At the same conference in 1990, Donkin stated another conjecture that encompasses Jantzen's Question, and connects the existence of good $(p,r)$-filtrations with tensoring by the 
$r$-th Steinberg representation. 

\begin{conj}\label{C:TMC} Let $G$ be a semisimple algebraic group over $k$, and $M$ be a rational $G$-module. 
\begin{itemize}
\item[(a)] $\operatorname{\sf [DFilt \Rightarrow]}$ If $M$ has a good $(p,r)$-filtration, then $\St_r\otimes M$ has a good filtration. 
Equivalently, for $\la \in X_r$, $\St_r\otimes L(\la)$ is a tilting module.
\item[(b)] $\operatorname{\sf [DFilt \Leftarrow]}$ If $\St_r\otimes M$ has a good filtration, then $M$ has a good $(p,r)$-filtration. 
\item[(c)] $\operatorname{\sf [DFilt \Leftrightarrow]}$ $M$ has a good $(p,r)$-filtration if and only if $\St_r\otimes M$ has a good filtration. 
\end{itemize} 
\end{conj}  

Recent progress has been made on connecting these various problems. We present the following diagram which indicates the various implications
between these questions and conjectures.  Symbolically, in the case of the two questions ([J-Quest($\dagger$)] and [HV-Quest]), we mean an {\em affirmative answer} to the question.

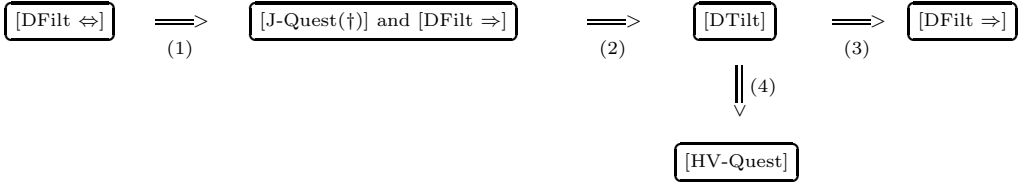
\begin{figure}[ht]\label{proj1} 
\setlength{\unitlength}{.5cm}
\begin{center}
\begin{picture}(14,4)
\tiny{\put(-7.0,3.75){\Ovalbox{[DFilt\ $\Leftrightarrow$]}}
\put(-0.5,3.75){\Ovalbox{[J-Quest($\dagger$)] and  [DFilt\ $\Rightarrow$]}}
\put(11.3,3.75){\Ovalbox{[DTilt]}}
\put(10.8,0.0){\Ovalbox{[HV-Quest]}}
\put(17.0,3.75){\Ovalbox{[DFilt\ $\Rightarrow$]}}

\put(-3,3.9){\line(1,0){1}}
\put(-3,3.8){\line(1,0){1}}
\put(-2,3.75){$>$}
\put(-2.7,3){$(1)$}

\put(8.5,3.9){\line(1,0){1}}
\put(8.5,3.8){\line(1,0){1}}
\put(9.5,3.75){$>$}
\put(8.8,3){$(2)$}

\put(15,3.9){\line(1,0){1}}
\put(15,3.8){\line(1,0){1}}
\put(16,3.75){$>$}
\put(15.3,3){$(3)$}

\put(12.45,2.7){\line(0,-1){1}}
\put(12.6,2.7){\line(0,-1){1}}
\put(12.35,1.35){$\vee$}
\put(12.8,2){$(4)$}

}

\end{picture}
\end{center}
\end{figure}

The first implication (1) follows from the fact that the tensor product of two modules with each admitting a  good filtration has a good filtration. The implication (4)  is clear. The implications (2) and (3) are deeper facts: (2) follows by \cite[Theorem 4.4.1]{BNPS22a} and (3) from \cite[Theorem 9.4.1]{KN15}. This picture has evolved over the past 20 years and includes earlier work that can be found in 
\cite{And01} and \cite{So18}. 

\subsection{Bounds on Jantzen's Question} Although this paper mainly focuses on determining new generic lower bounds on [DTilt], [HV-Quest], and [DFilt $\Rightarrow$], 
one should not lose sight of the current status on the bounds on [DFilt $\Leftarrow$] and [J-Quest]. 

First note that an affirmative answer to [J-Quest] is a special case of [DFilt $\Leftarrow$]. Andersen \cite[Theorem 1.1]{And19} proved that
[J-Quest] has an affirmative answer for $p\geq h(h-2)$ (a quadratic bound). Parshall and Scott \cite[Theorem 5.1]{PS15} have shown a linear bound, namely, [J-Quest] has a positive answer for $p\geq 2h-2$ as long as the Lusztig Character Formula holds for 
$G$. There are examples where [J-Quest] has a negative answer for primes less than $h$ (see \cite[Theorem 1.2.1]{BNPS22b}). An interesting problem would be to determine if a positive answer to [J-Quest] holds for $p>h$. 

Almost nothing is known about the general problem for verifying [DFilt $\Rightarrow$]. The negative examples for [JQuest] are clearly counterexamples for [DFilt $\Rightarrow$]. 
An important open question is to determine an effective lower bound on $p$ for the validity of [DFilt $\Rightarrow$]. 


\section{Tilting Module Conjecture: Lowering the Bound}\label{S:bound}

In this section, we lower the bound for the Tilting Module Conjecture to $p\geq 2h-4$. The proof given by Jantzen \cite{Jan80} for $p\geq 2h-2$ cannot be effectively used. One must rely on several of the new ideas developed by the authors in \cite[Section 3]{BNPS22a} to lower the bound.

\subsection{Divisibility by the Steinberg.} We first consider the characters of tilting and injective modules when divided by the 
character of the Steinberg module. Following the notation and the discussion in \cite[\S 3.1]{So20}, define for $\la \in X_1$
$$q(\la)= \text{ch} (\hQ_1((p-1)\rho+w_0 \la))/\chi((p-1)\rho),$$
$$t(\la) =\text{ch} (T((p-1)\rho+\la))/\chi((p-1)\rho).$$
For $\mu \in X^+,$  set $s(\mu) = \sum_{w \in W}e(w\mu)$ (where each distinct $w\mu$ appears only once),
which allows us to define $a_{\mu}^{\la}$ and $b_{\mu}^{\la}$ via 
\begin{equation}
q(\la) = \sum_{\mu \in X^+} a_{\mu}^{\la} s(\mu) 
\end{equation}
\begin{equation} 
t(\la) = \sum_{\mu \in X^+} b_{\mu}^{\la} s(\mu).
\end{equation}
The validity of these definitions is based on the fact that, for $\la \in X_1$, both $T((p-1)\rho +\la)$ and $\hQ_1((p-1)\rho +w_0\la),$ when viewed as $G_1T$-modules, have a filtration with factors of the form $\hZ_1'(\ga)$, with $\ga \in X$. One could equivalently define
$$a_{\mu}^{\la}= [\hQ_1((p-1)\rho +w_0\la):\hZ'_1((p-1)\rho + \mu)]$$
and
$$b_{\mu}^{\la}= [T((p-1)\rho +\la):\hZ'_1((p-1)\rho + \mu)].$$
The following is an immediate consequence of  \cite[Thm. 1.1]{So20}.

\begin{prop}\label{a=b} The Tilting Module Conjecture holds for $G$ if and only if for all $\la \in X_1$ and $\mu \in X^+$ with $\mu- \rho \uparrow \la-\rho, $ $a_{\mu}^{\la} = b_{\mu}^{\la}.$
\end{prop}

\subsection{Reduction to Levi Subgroups} In this section, one can show that the conditions in Proposition \ref{a=b} for [DTilt] has a natural extension to Levi factors. 
The following facts were first observed by Donkin   in \cite[Proposition 2.7]{Don93} and  \cite[Proposition 1.5 (ii)]{Don93}, respectively. For more details, the reader is referred to \cite[Sections 2.5, 2.6]{BNPS22b}.

 Let $L_J$ be a Levi subgroup of $G$ and  $\la \in X_r.$  Then we have an equality of $(L_J)_rT$-modules.
\begin{equation}\label{QLevi}
\hQ_{J,r}((p^r-1)\rho + w_{J,0}\la) = \bigoplus_{\nu \in \mathbb{N}J} \widehat{Q}_r((p^r-1)\rho+w_0 \la)_{(p^r-1)\rho+\la-\nu}.
\end{equation}
The  indecomposable tilting modules behave nicely when restricted to Levi subgroups. More precisely,  for any $\la \in X^+$, one obtains:
\begin{equation}\label{TLevi}
T_J( (p^r-1)\rho +\la)=\bigoplus_{\nu \in \mathbb{N}J} T( (p^r-1)\rho+\la)_{(p^r-1)\rho+\la-\nu}.
\end{equation}

Once again following the set-up in  \cite{So20}, define for  $\la \in (X_{J})_1$ 
$$q_J(\la)= \text{ch} (\hQ_{J,1}((p-1)\rho+w_{J,0} \la))/\chi_J((p-1)\rho),$$ 
$$t_J(\la) =\text{ch} (T_J((p-1)\rho+\la))/\chi_J((p-1)\rho).$$
Furthermore, for $L_J$ and $\mu \in (X_J)_+$,  define $s_J(\mu) = \sum_{w \in W_J}e(w\mu)$.
This yields $(a_{J})_{\mu}^{\la}$ and $(b_{J})_{\mu}^{\la}$ via
\begin{equation}
q_J(\la) = \sum_{\{\mu \in X^+ | \la-\mu \in \mathbb{N}J\}} (a_{J})_{\mu}^{\la} \;s_J(\mu) 
\end{equation}
\begin{equation}
 t_J(\la) = \sum_{\{\mu \in X^+ | \la-\mu \in \mathbb{N}J\}}  (b_{J})_{\mu}^{\la} \; s_J(\mu).
\end{equation} 
Here, the $(L_J)_1T$-modules $T_J((p-1)\rho +\la)$ and $\hQ_{J,1}((p-1)\rho +w_{J,0}\la)$ have filtrations with factors of the form $\hZ'_{J,1}(\ga)$, with $\ga \in X$.  It follows that 
$$(a_J)_{\mu}^{\la}= [\hQ_{J,1}((p-1)\rho +w_{J,0}\la):\hZ'_{J,1}((p-1)\rho + \mu)]$$
and
$$(b_J)_{\mu}^{\la}= [T_J((p-1)\rho +\la):\hZ'_{J,1}((p-1)\rho + \mu)].$$
Clearly,  $$\bigoplus_{\nu \in \mathbb{N}J} \hZ'_1((p-1)\rho+ \mu)_{(p-1)\rho+\la-\nu}
 = \begin{cases} \hZ'_{J,1}((p-1)\rho+ \mu) & \text{ if } \la - \mu \in \mathbb{N}J,\\
 0 & \text{ else. } 
 \end{cases}$$
Equations \eqref{QLevi} and \eqref{TLevi} now give rise to the following proposition.
\begin{prop}\label{RedLevi}
If the Tilting Module Conjecture holds for $L_J$ and $\la - \mu \in \mathbb{N}J$, then
$$a_{\mu}^{\la}=(a_{J})_{\mu}^{\la} \;=(b_{J})_{\mu}^{\la} \;=b_{\mu}^{\la}.$$
\end{prop}

\subsection{Minimal Counterexample}
For a fixed prime $p$ we say that $(G, p)$ is a {\em minimal counterexample for $p$} if the Tilting Module Conjecture does not hold for the pair $(G, p)$ but does hold for all pairs $(L_J, p),$ where $L_J$ denotes the Levi subgroup corresponding to a proper subset $J$ of $\Delta.$ 
\begin{prop}\label{min:count} Let $(G,p)$ be a minimal counterexample for the prime $p.$ If $a_{\mu}^{\la} \neq b_{\mu}^{\la}$ for
$\la \in X_1$ and $\mu \in X^+ $ with $\mu- \rho \uparrow \la-\rho,$
then
$\langle \mu, \alpha_0^{\vee} \rangle <  \langle \la, \alpha_0^{\vee} \rangle.$
\end{prop}
\begin{proof} 
Assume that $\la \in X_1$ and $\mu \in X^+, $ with $\mu- \rho \uparrow \la-\rho$ and $a_{\mu}^{\la} \neq b_{\mu}^{\la}.$ Note that any semisimple root system contains at least one simple root $\beta$  with $\langle \beta, \alpha_0^{\vee} \rangle > 0.$ Choose a $\beta$ and set  $J = \Delta-\{\beta\}.$ It follows from the fact that $G$ is a minimal counterexample and Proposition \ref{RedLevi} that $\la - \mu \notin {\mathbb N}J.$  Hence, $\la - \mu$ is a sum of positive roots that contains at least one copy of $\beta$, thus $\langle \la - \mu, \alpha_0^{\vee} \rangle > 0$ and the assertion follows. 
\end{proof}

\subsection{} We will now focus on the case when $p=2h-3$ and show that $T((p-1)\rho + \la)$ is indecomposable as a $G_{1}T$-module whenever  $\la \in X_1$ and  $\langle \la ,\alpha_0^{\vee} \rangle \leq p(h-2)$.  Note that by linkage properties, $b_{\mu}^{\la} \neq 0$ if and only if $\mu - \rho \uparrow \la - \rho$ (see \cite[Theorem 1.1]{So20}).  Recall also that $\hQ_1((p-1)\rho + w_0\la)$ is necessarily a $G_1T$-summand of $T((p-1)\rho + \la)$ (the crux of the TMC being that it is the {\em only} summand).


\begin{prop}\label{p(h-2)} Let $p=2h-3.$ If  $\la \in X_1$ satisfies  $\langle \la ,\alpha_0^{\vee} \rangle \leq p(h-2),$ then $T((p-1)\rho + \la)$ is indecomposable as a $G_1T$-module. In particular, $a_{\mu}^{\la} = b_{\mu}^{\la}$ for all $\mu \in X^+.$ 
\end{prop} 

\begin{proof} Note that for the stronger condition $\langle \la ,\alpha_0^{\vee} \rangle < p(h-2)$ the statement of the proposition is simply a paraphrasing of \cite[Theorem 3.4.3]{BNPS22a}. We may therefore assume that 
 $\la \in X_1$ with  $\langle \la ,\alpha_0^{\vee} \rangle = p(h-2)$. 
 
Suppose that there exists $\mu \in X^+$ with $\mu- \rho \uparrow \la-\rho$ but $a_{\mu}^{\la} \neq b_{\mu}^{\la}.$   Assume further that $\mu$ is a maximal weight satisfying these conditions. Clearly  $\mu < \la.$ 
Note that $(G, 2h-3)$ would be a minimal counterexample for $p=2h-3,$ because $h_J< h$ for any proper Levi subgroup $L_J.$ Thus $p \geq 2h_J-2,$ thereby satisfying the Jantzen condition and implying that [DTilt]  holds for $L_J$.  It follows from Proposition \ref{min:count} that $\langle \mu, \alpha_0^{\vee} \rangle <  \langle \la, \alpha_0^{\vee} \rangle=p(h-2).$

Since $a_{\mu}^{\la} \neq b_{\mu}^{\la}$, the tilting module $T((p-1)\rho + \la)$ decomposes as a $G_1T$-module. 
Moreover, due to the maximality of $\mu$, it must have a $G_1T$-summand with highest weight $(p-1)\rho + \mu.$ If we express $\mu$ in the form $\mu = \mu_0 + p \mu_1,$ with $\mu_0 \in X_1$ and $\mu_1 \in X^+,$ then $T((p-1)\rho + \la)$ has  $\hQ_1((p-1)\rho+w_0\mu_0) \otimes p\mu_1$ as a $G_1T$-summand. The  multiplicity is given by $b^{\la}_{\mu}-  a^{\la}_{\mu}.$ 

One concludes from the maximality of $\mu$ that $\Hom_{G_1}(L((p-1)\rho+w_0\mu_0), T((p-1)\rho + \la))$ has $p\mu_1$ as a highest weight. From $p \langle \mu_1, \alpha_0^{\vee} \rangle \leq  \langle \mu, \alpha_0^{\vee} \rangle<p(h-2)$, one obtains that 
$\langle \mu_1+\rho, \alpha_0^{\vee} \rangle \leq (h-2) +(h-1)=p.$ That is, $\mu_1$ lies in the closure of the fundamental alcove.  Hence, for any simple composition factor $L(\s)$ of $\Hom_{G_1}(L((p-1)\rho+w_0\mu_0), T((p-1)\rho + \la))^{(-1)}$, the dominant weight $\s$ lies in the closure of the fundamental alcove.   As there are no extensions between such modules, 
 $\Hom_{G_1}(L((p-1)\rho+w_0\mu_0), T((p-1)\rho + \la))^{(-1)}$ is completely reducible as a $G$-module and 
$$\dim \Hom_{G}(L((p-1)\rho+w_0\mu_0)\otimes L(\mu_1)^{(1)}, T((p-1)\rho + \la))=b^{\la}_{\mu}-  a^{\la}_{\mu}.$$ Therefore,  $\hQ_1((p-1)\rho+w_0\mu_0)\otimes L(\mu_1)^{(1)}$ appears $b^{\la}_{\mu}-  a^{\la}_{\mu}$ many times as a $G_1T$-summand of $T((p-1)\rho + \la).$

Since $\langle \mu, \alpha_0^{\vee} \rangle<p(h-2)$ and $\langle \la, \alpha_0^{\vee} \rangle =p(h-2),$ it follows from  \cite[Lemma 3.4.2]{BNPS22a}  that $T((p-1)\rho + \mu) \cong \hQ_1((p-1)\rho+w_0\mu_0)\otimes L(\mu_1)^{(1)}$ as a $G_1T$-module and that $T((p-1)\rho + \mu)$ is the injective hull and projective cover of $L((p-1)\rho+w_0\mu_0)\otimes L(\mu_1)^{(1)}$ in the truncated category $\Mod((p-1)\rho + \la)$\footnote{For $\s \in X^+$, $\Mod(\s)$ denotes the category of all finite-dimensional $G$-modules whose highest weight is less than or equal to $\s$}. But now the argument given at the end of the proof of \cite[Proposition 5.6.2]{BNPS22a} shows that the $b^{\la}_{\mu}-  a^{\la}_{\mu}$ copies of  $T((p-1)\rho + \mu)$ are actually $G$-summands not just $G_1T$-summands of  $T((p-1)\rho + \la)$, a contradiction.

\end{proof}

\subsection{} We can now lower the bound on the Tilting Module Conjecture and verify the statements of Theorem~\ref{T:main}(a) and (b). 

\begin{theorem}\label{T:lowerboundtilting} Let $G$ be a simple algebraic group and $p\geq 2h-4$. Then the Tilting Module Conjecture holds. 
\end{theorem} 

\begin{proof} The case when $p\geq 2h-2$ is handled in \cite{Jan80} (see also \cite[Corollary 3.3.2]{BNPS22a} for an alternative proof). 

First assume that $p=2h-3.$ Note that, for  $\la \in X_1$, $\langle \la ,\alpha_0^{\vee} \rangle \leq p(h-2)$ if and only if $\ga := (p-1)\rho +w_0 \la$ is a $p$-restricted weight that is not contained in the interior of the lowest alcove. It follows from Proposition \ref{p(h-2)} that the TMC holds for all $\hQ_1(\ga)$ as long as $\ga$ is not in the interior of the lowest alcove.  However, one may clearly choose $\la \in X_1$ with $\langle \la ,\alpha_0^{\vee} \rangle = p(h-2)$. Then $\ga := (p-1)\rho +w_0 \la$ is a $p$-restricted weight on the upper wall of the lowest alcove (so the TMC holds for $\hQ_1(\ga)$). For any $\nu$ in the interior of the lowest alcove the translation principle   \cite[II.11.10, II.E.11]{rags} now says that one has the following isomorphisms as $G_{1}T$-modules: 
$$T(2(p-1)\rho+ w_0 \nu)\cong T_{\ga}^{\nu}[T(2(p-1)\rho+w_0\ga)]\cong T_{\ga}^{\nu}[\hQ_1(\ga)]= \hQ_1(\nu).$$
This completes the $p=2h-3$ case. 

The only situation in which we have $p = 2h - 4$ is when $p = 2$ and $h = 3$. That is, when the root system is of type $A_2$.  [DTilt]  is known to hold (cf. \cite{Don17}) in type $A_2$ (for all primes), and so we may claim that 
[DTilt]  holds for $p\geq 2h-4.$
\end{proof}


\section{Extensions}\label{S:Ext}

\subsection{} A key observation in \cite{BNPS22b} is that counterexamples to [DTilt] seem to be present whenever one can identify some $\la \in X_1$ for which $\opH^1(G_1,L(\la))^{(-1)}$ is not a tilting module over $G$.  More generally, one may consider the question of whether $\Ext_{G_1}^1(L(\la),L(\mu))^{(-1)}$ is tilting for $\la, \mu \in X_1$.  In Theorem~\ref{T:Ext} below, it is shown that, under the condition $p \geq 2h - 4$ where the TMC is known to hold, such Ext-groups are indeed tilting.

\subsection{Weights in $\operatorname{Ext}^{1}$} Combining work of Andersen \cite{And84} and Bendel, Nakano, and Pillen \cite{BNP04}, one can show that weights of $\Ext_{G_1}^1(L(\la),L(\mu))^{(-1)}$ are ``small.''  Indeed, the following statement is an improvement of \cite[Proposition 5.2]{BNP04} that gave a bound of $h - 1$ in all types. 

\begin{prop}\label{P:bound} Assume $p > 2$ if $\Phi$ has two root lengths and $p > 3$ if $\Phi$ is of type $G_2$.  Let $\la, \mu \in X_1$ and $\ga$ be a dominant weight of $\Ext_{G_1}^1(L(\la),L(\mu))^{(-1)}$. Then 
$$
\langle \ga, \a_0^{\vee}\rangle \leq
\begin{cases}
h - 2 &\text{ if } \Phi \neq A_1,\\
h - 1 &\text{ if } \Phi = A_1.
\end{cases}
$$
\end{prop} 

\begin{proof}  By \cite[Lemma 2.3]{And84}, there exists a simple root $\a$ such that
\begin{equation}\label{E:And}
p\ga \leq -w_0\la + \mu + \a.
\end{equation}
The statement in \cite{And84} assumed $p > h$ as the proof made use of a general cohomology result of Andersen and Jantzen \cite{AJ}.  For our purposes, we only need information on $\opH^1(G_1,H^0(\omega))$ for  $\omega \in X^+$ that can be obtained from \cite[Proposition 4.1]{Jan91} (with the given conditions on the prime $p$).    
From the proof of \cite[Proposition 5.2]{BNP04}, one also has (with no assumption on the prime)
\begin{equation}\label{E:BNP}
p\ga \leq 2(p-1)\rho + w_0\mu - \la.
\end{equation}
Adding \eqref{E:And} and \eqref{E:BNP} gives
\begin{equation}\label{E:combined}
2p\ga \leq 2(p-1)\rho + \mu + w_0\mu - w_0\la - \la + \a.
\end{equation}

For $\s \in X$, 
$$
\langle w_0\s,\a_0^{\vee}\rangle = \langle \s,w_0^{-1}(\a^{\vee}_0)\rangle = \langle \s,-\a_0^{\vee}\rangle = -\langle\s,\a_0^{\vee}\rangle.
$$
Hence, taking the inner product with $\a_0^{\vee}$ on both sides of \eqref{E:combined} yields
$$
2p\langle\ga,\a_0^{\vee}\rangle \leq 2(p-1)\langle\rho,\a_0^{\vee}\rangle + \langle\a,\a_0^{\vee}\rangle.
$$
Rewriting gives
\begin{equation}\label{E:a0}
2p\langle\ga,\a_0^{\vee}\rangle \leq 2p\langle\rho,\a_0^{\vee}\rangle - \langle 2\rho - \a,\a_0^{\vee}\rangle.
\end{equation}
Since $2\rho$ is the sum of all the positive roots, $\langle 2\rho - \a,\a_0^{\vee}\rangle > 0$ except in type $A_1$, where it equals zero.
We conclude that 
$$
\langle\ga,\a_0^{\vee}\rangle \leq \langle\rho,\a_0^{\vee}\rangle = h - 1,
$$
with equality holding only in type $A_1$, giving the claim.
\end{proof}

\subsection{Structure of $\operatorname{Ext}^{1}$} With the result in the previous section, one can now investigate the structure of  $\Ext_{G_1}^1(L(\la),L(\mu))^{-1}$ for 
 $\la, \mu \in X_1$ and verify statement (d) of Theorem~\ref{T:main}. 

\begin{theorem}\label{T:Ext} Let $G$ be a simple algebraic group, $p \geq 2h - 4$, and $\la, \mu \in X_1$.  Then $\Ext_{G_1}^1(L(\la),L(\mu))^{(-1)}$ is a tilting module and completely reducible.
\end{theorem}

\begin{proof} Assume first that $p \geq 2h - 3$.  Observe that $p$ satisfies the conditions of Proposition \ref{P:bound}.   Assume further that $\Phi$ is not of type $A_1$.  Let $L(\ga)$ be a composition factor of $\Ext_{G_1}^1(L(\la),L(\mu))^{-1}$.  By Proposition \ref{P:bound}, 
$$
\langle \ga + \rho,\a_0^{\vee}\rangle = \langle\ga,\a_0^{\vee}\rangle + \langle\rho,\a_0^{\vee}\rangle \leq h - 2 + h -1 = 2h -3 \leq p.
$$
Therefore, $\ga$ lies in the closure of the fundamental alcove and
$$
L(\ga) = \nabla(\ga) = \Delta(\ga) = T(\ga).
$$
Furthermore, there are no extensions between such a pair of modules (cf.~\cite[Proposition II.4.13]{rags}). Hence, $\Ext_{G_1}^1(L(\la),L(\mu))^{(-1)}$ is a direct sum of indecomposable tilting modules, each of which is simple.

Consider the case that $\Phi$ is of type $A_1$ (i.e., $G$ is $SL_2$).  By \cite[Proposition 5.2]{BNP04}, one has
$$
\langle \ga + \rho,\a_0^{\vee}\rangle = \langle\ga,\a_0^{\vee}\rangle + \langle\rho,\a_0^{\vee}\rangle \leq h - 1 + h -1 = 2h -2 = 2 \leq p.
$$
So the result follows as above. One could also use explicit knowledge of the Ext-groups to obtain the conclusion.  

As noted in the proof of Theorem~\ref{T:lowerboundtilting}, if $p = 2h - 4$, then $\Phi = A_2$ and $p = 2$. For type $A_2$, Yehia \cite[Proposition 3.3.2]{Y82} computed $\Ext_{G_1}^1(L(\la),L(\mu))$ in all primes, from which the statement of the theorem follows (cf. \cite[Corollary 3.3.4]{Y82}). 

\end{proof} 

\section{Rank $2$ Analysis} \label{S:Rank2}

\subsection{} The goal of this section is to verify the statements in Theorem~\ref{T:main2}. Part (d) will be verified in Section~\ref{SS:ExtRank2}. We provide details and references for parts (a)-(c) below. 

Part (a) follows from \cite[Theorem 1.1.1]{BNPS22a} which verifies [DTilt] for all rank $2$ cases aside from $\Phi=\rm{G}_{2}$ and $p=2$. In the latter case, it was shown that there is a counterexample to 
[DTilt], see \cite[Theorem 4.4.1]{BNPS20}. Part (b) follows from part (a), although [HVQuest] is still unresolved for the case when $\Phi=\rm{G}_{2}$ and $p=2$. 

For part (c), Kildetoft and Nakano verified [DFilt $\Rightarrow$] for the rank $2$ cases except when $\Phi=\rm{G}_{2}$ and $p=7$ (see \cite[Section 8]{KN15}). This case was later verified by the authors in \cite[Theorem 4.4.1]{BNPS19}. 

\subsection{Structure of $\operatorname{Ext}^{1}$ in Rank $2$}\label{SS:ExtRank2}

In this subsection we justify statement (d) of Theorem~\ref{T:main2}.

\begin{theorem}\label{T:ExtRank2} Let $G$ be a simple algebraic group with underlying root system of rank $2$, and let $p>2$ if $\Phi$ is of type $G_2$.  Then $\Ext_{G_1}^1(L(\la),L(\mu))^{(-1)}$ is a tilting module and completely reducible for all $\la, \mu \in X_1$.
\end{theorem}

\begin{proof}
Applying Theorem~\ref{T:Ext}, we need only check the primes smaller than $2h-4$.

Type $A_2$  - Here $h=3$, so $2h-4=2$, and there is nothing to check. We note, however, that the proof of Theorem~\ref{T:Ext} handled this exact case by citing the work of Yehia \cite{Y82}.

Type $B_2$ - Here $h=4$, so $2h-4=4$, thus we must deal with $p=2,3$.  For $p = 2$, this follows from the computations in \cite[p. 1019]{Sin94}.   The only modules that arise in $\Ext_{G_1}^1(L(\la),L(\mu))^{(-1)}$ are $k$ and $L(\omega_2) = \nabla(\omega_2)$.  As both modules are tilting and simple, any module made up of such factors is completely reducible. Note that Sin is working with type $C_2$. 

For $p = 3$, we first apply Proposition \ref{P:bound} to see that if $\ga$ is a weight of $\Ext_{G_1}^1(L(\la),L(\mu))^{(-1)}$ then $\langle\ga,\a_0^{\vee}\rangle \leq h - 2 = 2$.  The only possible dominant weights that satisfy this bound are $0$, $\omega_1$, $\omega_2$, and $2\omega_2$.  In all cases, $L(\ga) = \nabla(\ga),$ see \cite{L}. This can also be verified via the sum formula of the Jantzen filtration \cite[Prop. II.8.19]{rags}.

Type $G_2$ - Here $h=6$, so $2h-4=8$, and we must check $p = 3,5,7$.  (The result is false for $p=2$, see \cite[p. 2633]{DS} and also \cite{Jan91}.)  
According to \ref{P:bound} one has $\langle\ga,\a_0^{\vee}\rangle \leq h - 2 = 4$. The options for $\ga$ are $0$, $\omega_1$, $\omega_2$, and $2\omega_1$.
For $p=3$, we find from \cite[p. 1022]{Sin94} that the only modules that arise in $\Ext_{G_1}^1(L(\la),L(\mu))^{(-1)}$ are $k$ and $L(2\omega_1) = \nabla(2\omega_1)$.  
For $p=5,$
the tables in \cite[\S 18]{Hum06} show that $L(\ga) = \nabla(\ga)$ for all those $\ga$.  
For $p = 7$, there is only one potential issue: $L(2\omega_1) \neq \nabla(2\omega_1)$.  
However, it follows from \cite[Remark 8.10]{BNPS22a} and the table in 
\cite[Section 4.2, Figure 3]{Lin}  
that only $k$, $L(\omega_1)$ and $L(\omega_2)$ appear as composition factors of 
$\Ext_{G_1}^1(L(\la),L(\mu))^{(-1)}.$ Moreover, all of these are tilting and appear as summands.  

\end{proof}

\providecommand{\bysame}{\leavevmode\hbox
to3em{\hrulefill}\thinspace}

\end{document}